\documentclass[twoside,11pt]{amsart}

\usepackage[english]{babel}
\usepackage[utf8x]{inputenc}
\usepackage{url}
\usepackage{amsmath}
\usepackage{graphicx}
\usepackage[colorinlistoftodos]{todonotes}
\usepackage{amsmath,amscd,amsthm}
\usepackage{amstext, amsfonts, a4}
\usepackage{amssymb}
\usepackage{tikz-cd}
\usepackage{fancyhdr}
\usepackage{enumerate}
\usepackage{xcolor}
\usepackage{cleveref}

\numberwithin{equation}{section}
\newtheorem{thm}[equation]{Theorem}
\newtheorem*{thm*}{Theorem}

\newtheorem{prop}[equation]{Proposition}
\newtheorem{cor}[equation]{Corollary}

\theoremstyle{definition}

\newtheorem{ex}[equation]{Example}

\newtheorem{qu}[equation]{Question}

\renewcommand{\dim}{\operatorname{\mathsf{dim}}}

\newcommand{\car}{\mathsf{char}}

\newcommand{\vf}{\varphi}
\newcommand{\mg}[1]{{#1}^{\times}}
\newcommand{\sq}[1]{{#1}^{\times 2}}
\newcommand{\scg}[1]{\mg{#1}/\sq{#1}}

\newcommand{\nat}{\mathbb{N}}

\newcommand{\la}{\langle}
\newcommand{\ra}{\rangle}
\newcommand{\lla}{\la\!\la}
\newcommand{\rra}{\ra\!\ra}

\renewcommand{\leq}{\leqslant}
\renewcommand{\geq}{\geqslant}

\newcommand{\tors}{\mathsf{tors}}

\newcommand\rk{\operatorname{\mathsf{rk}}}

\newcommand{\W}{\mathsf{W}}
\newcommand{\I}{\mathsf{I}}
\newcommand{\rr}{\mathbb{R}}

\newcommand{\mc}{\mathcal}

\newcommand{\ovl}{\overline}
\renewcommand{\setminus}{\smallsetminus}

\renewcommand{\dim}{\mathsf{dim}}
\renewcommand{\leq}{\leqslant}
\renewcommand{\geq}{\geqslant}
\renewcommand{\sup}{\mathsf{sup}}

\newcommand{\D}{\mathsf{D}}

\newcommand{\mf}{\mathfrak}
\newcommand{\mfm}{\mf{m}}

\newcommand{\qq}{\mathbb{Q}}

\newcommand{\wt}{\widetilde}

\renewcommand{\setminus}{\smallsetminus}

\makeatletter
\newcommand{\bigperp}{%
  \mathop{\mathpalette\bigp@rp\relax}%
  \displaylimits
}

\newcommand{\bigp@rp}[2]{%
  \vcenter{
    \m@th\hbox{\scalebox{\ifx#1\displaystyle2.1\else1.5\fi}{$#1\perp$}}
  }%
}
\makeatother

\title[Fields where torsion forms are strongly balanced]{Fields where torsion forms decompose} 

\date{13 May, 2026}

\author{M.~Archita}
\author{Karim Johannes Becher}

\address{University of Antwerp, Department of Mathematics, Antwerp, Belgium.}
\email{karimjohannes.becher@uantwerpen.be}
\email{archita.mondal@uantwerpen.be}

\begin{document}

\begin{abstract}
Over a real field which is an extension of transcendence degree $1$ of a hereditarily pythagorean base field, every quadratic form which is torsion decomposes into an orthogonal sum of $2$-dimensional torsion forms. 
This is obtained from a more general study of  weakly isotropic forms over henselian valued fields and over function fields in one variable.


\medskip\noindent
{\sc Classification (MSC 2020):} 11E04, 
11E81, 
12D15, 
12J10 
\end{abstract}

\maketitle

\section{Introduction}

Let $K$ be a field of characteristic different from $2$.
By the Artin-Schreier Theorem \cite[Chap.~VIII, Theorem 1.10]{Lam05}, $K$ can be endowed with a field ordering if and only if  $-1$ is not a sum of squares in $K$. In this case we call the field $K$ \emph{real}, otherwise \emph{nonreal}.

We denote by $\W K$ the Witt ring of quadratic forms over $K$.
A regular quadratic form over $K$ which corresponds to a torsion element of $\W K$ is called a \emph{torsion form}.
If $K$ is nonreal, then  $\W K$ is a torsion group, and hence every regular quadratic form over $K$ is torsion, and by diagonalization it decomposes into a sum of $1$-dimensional torsion forms. Assume that $K$ is real.
In that case, Pfister's Local-Global Principle \cite[Chap.~VIII, Theorem~3.2]{Lam05} characterizes torsion forms as those regular quadratic forms having signature zero at every field ordering of $K$.
In particular, torsion forms are even-dimensional. 
Two-dimensional quadratic forms are called \emph{binary forms}. 
In \cite{AP77}, J.K.~Arason and A.~Pfister called a quadratic form \emph{strongly balanced} (\emph{stark ausgeglichen}) if it is isometric to a sum of binary torsion forms.
They studied the problem over which real fields all torsion forms (\emph{weakly balanced} / \emph{schwach ausgeglichen}) are in fact strongly balanced.

In particular, they show that the quadratic form $$\la 1,Y,-(3+X^2+Y),-Y(1+X^2+3Y)\ra$$ over $\qq(X,Y)$ (the rational function field in two variables $X,Y$ over $\qq$) is torsion but not strongly balanced. 
See \Cref{AP:w3} for the more general statement on rational function fields in two variables. 
No example of such a quadratic form seems to be known over a field that is not a rational function field in two variables over some real field.

In particular, one may ask over which base fields $K$ every real field extension of transcendence degree $1$ has the property that all torsion forms are strongly balanced. 
When $K$ is a real number field, \cite[Satz 5]{AP77} establishes this property for the rational function field $K(X)$, but it is still open whether the same holds for its finite extensions.

    In this article, we settle this problem for a different class of base fields, confirming \cite[Conjecture 4.4]{Bec06} by proving the following.

\begin{thm*}
    Let $K$ be a hereditarily pythagorean field. Let $F/K$ be a field extension of transcendence degree $1$ such that $F$ is real. 
    Then every torsion form over $F$ is strongly balanced.
\end{thm*}

Recall a field is called \emph{pythagorean} if in it every sum of two squares is equal to a square,
and \emph{hereditarily
pythagorean} if it is real and all its finite real field extensions are pythagorean.

The theorem will be established by \Cref{main}, which states a formulation that is equivalent in view of \Cref{Arason-Pfister}.
We will use Br\"ocker's valuation theoretic characterization of hereditarily pythagorean fields \cite[Prop.~3.5]{Bro76} and derive our theorem from a more general relation for function fields in one variable over a henselian valued field.
Our proof follows the path of  \cite{BDGMZ23}, using the same methods. In particular, it relies on  \cite[Theorem 4.4]{BDGMZ23}, a variant of Mehmeti's local-global principle from \cite{Meh19} for quadratic forms over a function field in one variable with a henselian rank-$1$ valued base field.
For the present purpose, however, we need an extension covering infinite-dimensional quadratic spaces, which is provided by \Cref{Mehmeti-infinite} and of independent interest.

We will use basic facts and standard terminology from quadratic form theory and valuation theory, our standard references being \cite{Lam05} and \cite{EP05}, respectively.

\section{Minimal weakly isotropic forms}

\label{S:mwif}

By a \emph{quadratic space over $K$}, we mean a pair $\vf=(V,q)$ of a $K$-vector space $V$ 
and a map $q:V\to K$ such that 
$V\times V\to K,(x,y)\mapsto q(x+y)-q(x)-q(y)$ is $K$-bilinear and  $q(\lambda x)=\lambda^2 x$ for every $\lambda\in K$ and every $x\in V$.
A \emph{quadratic form} is a finite-dimensional quadratic space, and if the corresponding $K$-bilinear form is non-degenerate, it is said to be \emph{regular}.

Let $\vf=(V,q)$ be a quadratic space over $K$.
We call $\vf$ \emph{isotropic} if $q(v)=0$ for some $v\in V\setminus\{0\}$, and \emph{anisotropic} otherwise.
    For a field extension $L/K$, $q$ extends uniquely to a map $q_L: V_L\to L$ on the $L$-vector space $V_L=V\otimes_KL$ that makes $(V_L,q_L)$ into a quadratic space over $L$, which we denote by $\vf_L$.
    
Let $\vf$ be a regular quadratic form over $K$. For $n\in\nat$, we denote by $n\times \vf$ the $n$-fold orthogonal sum $\vf\perp\ldots\perp\vf$. We call $\vf$ \emph{weakly isotropic} if $n\times \vf$ is isotropic for some $n\in\nat$, and we call $\vf$ \emph{minimal weakly isotropic} if it is weakly isotropic but contains no proper weakly isotropic subform over $K$.
We denote by $\dim(\vf)$ the dimension of $\vf$.
In \cite{Bec06}, the second author of the present article introduced the following field invariant:
$$w(K)=\sup \{\dim(\vf)\mid\vf \mbox{ minimal weakly isotropic form over }K\}\in \nat \cup \{\infty\}\,.$$
We continue the study of this invariant.

\begin{prop}\label{nr}
We have $w(K)=1$ if and only if $K$ is nonreal.
\end{prop}
\begin{proof}
There exists a $1$-dimensional weakly isotropic quadratic form over $K$ if and only if $-1$ is a sum of squares, that is, if and only if $K$ is nonreal, and in that case, any $1$-dimensional quadratic form over $K$ is weakly isotropic.
\end{proof}

For a quadratic form $\vf$ over $K$, we denote by $\D(\vf)$ the set of nonzero elements represented by $\vf$.
For $n\in\nat$ and $a_1,\dots,a_n\in\mg{K}$, we write $\D\la a_1,\dots,a_n\ra$ instead of $\D(\la a_1,\dots,a_n\ra)$.
We further denote by $\mg{(\mathsf{\Sigma}K^2)}$ the subgroup of $\mg{K}$ consisting of the nonzero sums of squares in $K$.
A quadratic form $\vf$ over $K$ is called \emph{totally positive} if it is regular and $\D(\vf)\subseteq \mg{(\mathsf{\Sigma}K^2)}$.

The following statement generalizes \cite[Satz~1, Folgerung]{AP77}, which corresponds to the case where $r=2$.

\begin{prop}\label{AP:B}
        Let $r\in\nat$ with $r\geq 2$. Assume that, for every $t_0,\dots,t_r\in\mg{(\mathsf{\Sigma}K^2)}$, there exists a totally positive $r$-dimensional quadratic form $\rho$ over $K$ with $t_0,\dots,t_r\in\D(\rho)$.
    Then $w(K)\leq r$.
\end{prop}
\begin{proof}
In view of \Cref{nr} we may assume that $K$ is real.
Let $\vf$ be a weakly isotropic quadratic form over $K$ with $\dim(\vf)>r$. 
We need to show that $\vf$ is not minimal weakly isotropic.

Note that the hypothesis of the statement remains true when the value of $r$ is increased.
Therefore we may assume that $\dim(\vf)=r+1$.
Let $a_0,\dots,a_r\in\mg{K}$ be such that $\vf\simeq \la a_0,\dots,a_r\ra$.
Since $\vf$ is weakly isotropic, there exist $t_0,\dots,t_r\in\mg{(\mathsf{\Sigma}K^2)}$ such that $\la a_0t_0,\dots,a_rt_r\ra$ is isotropic.
By the hypothesis, there exists a totally positive $r$-dimensional quadratic form $\rho$ over $K$ such that $t_0,\dots,t_r\in\D(\rho)$.
Let $s_1,\dots,s_r\in\mg{K}$ be such that $\rho\simeq \la s_1,\dots,s_r\ra$.
Since $\rho$ is totally positive, we have that $s_1,\dots,s_r\in\mg{(\mathsf{\Sigma}K^2)}$.
Since $\rho\otimes\vf$ is isotropic, there exist $b_1,\dots,b_r\in\D(\vf)$ such that $\la s_1b_1,\dots,s_rb_r\ra$ is isotropic.
Then $\vf$ has an $r$-dimensional subform $\vf'$ such $b_1,\dots,b_r\in\D(\vf')$.
It follows that $\rho\otimes\vf'$ is isotropic.
Since $\rho$ is totally positive, we obtain that $\vf'$ is weakly isotropic.
Therefore $\vf$ is not minimal weakly isotropic.
\end{proof}

We denote by $p(K)$ the smallest $n\in\nat$ such that every sum of squares in $K$ can be written as a sum of $n$ squares, and we set $p(K)=\infty$ if no such integer $n\in\nat$ exists.
This field invariant is called the \emph{Pythagoras number of $K$}.
Hence $K$ is  pythagorean if and only if $p(K) = 1$.
We retrieve \cite[Theorem 3.3]{Bec06}. 

\begin{cor}\label{wp}
If $K$ is not real pythagorean, then $w(K)\leq p(K)$. 
\end{cor}
\begin{proof}
We may assume that $1<p(K)<\infty$, and then apply \Cref{AP:B} with $r=p(K)$ and $\rho=r\times \la 1\ra$. 
\end{proof}

\begin{prop}\label{Pfister-balanced}
    Let $\pi$ be a Pfister form over $K$.
    Then $\pi$ is strongly balanced 
    if and only if there exists $n\in\nat$ with $n< p(K)$ such that $n\times\la 1\ra\perp\pi$ is isotropic.
\end{prop}
\begin{proof}
Assume that $\pi$ is isometric to an orthogonal sum of binary torsion forms. 
In particular, $\pi$ contains some binary torsion form $\beta$. 
Let $a,d\in\mg{K}$ be such that $\beta\simeq a \la 1, -d \ra$ over $K$. 
Then $a\in\D(\beta)\subseteq\D(\pi)$. Since $\pi$ is a Pfister form, it follows that $a\pi\simeq \pi$, by \cite[Chap.~X, Theorem~2.8]{Lam05}.
Since $\la 1,-d\ra\simeq a\beta$, which is a subform of $a\pi\simeq\pi$, 
we conclude that $\la 1,-d\ra$ is a subform of $\pi$.
As $\beta$ is torsion, we have $d\in\mg{(\mathsf{\Sigma}K^2)}$.
Hence $d$ is a sum of $n+1$ squares in $K$ for some $n\in\nat$ with $n< p(K)$.
Then the form $(n+1)\times \la 1\ra\perp\la -d\ra$ over $K$ is isotropic.
As $\la 1,-d\ra$ is a subform of $\pi$, it follows that $n\times \la 1\ra\perp\pi$ is isotropic.

Conversely, assume there exists $n\in\nat$ with $n< p(K)$ such that $n\times\la 1\ra\perp\pi$ is isotropic. Let $\pi'$ be the quadratic form over $K$ such that $\pi\simeq \la 1\ra\perp\pi'$.
Then $(n+1)\times \la 1\ra\perp\pi'$ is isotropic.
Therefore $-\pi'$ and $(n+1)\times\la 1\ra$ represent some element $d\in \mg{K}$ in common.
Then $d$ is a sum of $n+1$ squares in $K$ and $\la 1,-d\ra$ is a subform of $\pi$.
It follows that $\pi_{K(\sqrt{d})}$ is isotropic, and hence hyperbolic, because $\pi$ is a Pfister form. Using \cite[Chap.~VII, Theorem~3.2]{Lam05}, we obtain that $\pi\simeq \lla d \rra \otimes \psi$ for some quadratic form $\psi$ over $K$.
For $r=\dim(\psi)$ and $a_1,\dots,a_r\in\mg{K}$ with $\psi\simeq \la a_1,\dots,a_r\ra$, we obtain that 
$$\pi\simeq \lla d\rra\otimes\psi\simeq \la a_1,-a_1d\ra\perp\ldots\perp\la a_r,-a_rd\ra\,,$$ 
and since $d$ is a sum of squares, this shows that $\pi$ is strongly balanced.
\end{proof}

\begin{cor}[Arason-Pfister]\label{Arason-Pfister}
        We have $w(K)\leq 2$ if and only if, for every torsion $2$-fold Pfister form $\pi$ over $K$, there exists $n\in\nat$ with $n<p(K)$ such that $n\times\la 1\ra\perp\pi$ is isotropic.
        Furthermore, $w(K)=2$ if and only if $K$ is real and every torsion form over $K$ is strongly balanced.
\end{cor}
    \begin{proof}
        This follows from \cite[Satz 4]{AP77} together with  \Cref{Pfister-balanced}.
    \end{proof}

In view of \Cref{Arason-Pfister}, the  example by Arason and Pfister mentioned in the introduction shows that $w(\qq(X,Y))\geq 3$.
It is plausible that there exist real fields with arbitrarily large (or infinite) $w$-invariant, but so far there is no confirmed example of a field $K$ with $w(K)>3$.

For $n\in\nat$, we denote by $\I^n K$ the $n$th power of the fundamental ideal $\I K$ in $\W K$, the Witt ring of quadratic forms over $K$.
We further denote by $\I_\tors^n K$ the torsion part of $\I^n K$.

\begin{prop}\label{In-tors-w}
    If $n\in\nat$ is such that $\I^{n+2}_\tors K=0$, then $w(K)\leq 2^{n}$.
\end{prop}
\begin{proof}
    See \cite[Theorem 3.5]{Bec06}.
\end{proof}

Note that $(\I_\tors K)^n\subseteq \I_\tors^n K$ for every $n\in\nat^+$.

\begin{thm}[Arason-Pfister]\label{AP:w3}
    Let $k$ be a real field  with $(\I_\tors k)^2\neq 0$.
    Then there exist $u_1,u_2\in \mg{(\mathsf{\Sigma}k^2)}$ such that $u_2\notin\D\la 1,u_1\ra$. For any such elements $u_1$ and $u_2$, the quadratic form $\lla -Y,u_1+X^2+u_2Y\rra$ over $K=k(X,Y)$ is torsion but not strongly balanced.
    In particular $w(K)\geq 3$.
\end{thm}
\begin{proof}
    Since $(\I_\tors k)^2\neq 0$, there exist
    $u,u'\in \mg{(\mathsf{\Sigma}k^2)}$ such that $\la 1,-u,-u',uu'\ra$ is anisotropic.
    For $u_1=uu'$ and $u_2=u$, we obtain that $u_1,u_2\in\mg{(\mathsf{\Sigma}k^2)}$ and $u_2\notin \D\la 1,u_1\ra$.
    It follows by \cite[Satz 3]{AP77} that 
    there exists no $t\in\mg{(\mathsf{\Sigma}k(X)^2)}$ such that $X^2+u_1$ and $u_2$
    are represented by the binary form $\la 1,t\ra$ over $k(X)$.
    Therefore \cite[Satz 2]{AP77} shows that $\pi$ has the claimed properties.
\end{proof}

\begin{ex}
 Real fields $k$ with $u_1,u_2\in\mg{(\mathsf{\Sigma}k^2)}$ such that $u_2\notin\D\la 1,u_1\ra$ are:
 \begin{enumerate}[$(i)$]
     \item Any real field $k$ with $p(k)\geq 3$. (Take $u_1=1$ and $u_2\in\mg{(\mathsf{\Sigma}k^2)}$ not a sum of two squares in $k$.)
     \item Any real number field. (Special case of $(i)$, since $p(k)\in\{3,4\}$, by \cite[Chap.~XI, Example 5.9~(2)]{Lam05}.)
     \item The fields $\rr(t_1,t_2)$, $\rr(\!(t_1)\!)(t_2)$ and $\rr(\!(t_1,t_2)\!)$. (Take $u_1=t_1^2+t_2^2$ and $u_2=t_1^2(t_1+1)^2+t_2^2(t_1-1)^2$; see \cite[Example 6.6]{BL11}.)
 \end{enumerate}
 In these cases, we obtain for $K=k(X,Y)$ that $w(K)\geq 3$, by \Cref{AP:w3}, and  
 furthermore $\I^4_\tors K\neq 0$ and $p(K)\geq 5$.
\end{ex}

We list a few crucial open questions in the study of the invariant $w$.

\begin{qu}
    Does there exist a field $K$ with $w(K)\geq 3$ and $p(K)\leq 4$?
\end{qu}
\begin{qu}
    Does there exist a field $K$ with $w(K)\geq 3$ and $\I^4_\tors K=0$?
\end{qu}

\begin{qu}
Does there exists a field $K$ with $w(K)>3$?    
\end{qu}

\begin{qu}
    Is $w(K)\leq 2$ if $K/\qq$ is an extension of transcendence degree~$1$?
\end{qu}

\section{Valuations and weakly isotropic forms}

We now study the interplay between valuations and the $w$-invariant.

Consider a valuation $v$ on $K$.
We denote by $vK$ its value group, and we call $v$ \emph{trivial} if $vK=0$.
We denote by $\mc{O}_v$ the valuation ring of $v$ and by $\mfm_v$ its maximal ideal.
We set $Kv=\mc{O}_v/\mfm_v$ and call it the \emph{residue field of $v$}.
For $a\in\mc{O}_v$ we write $\ovl{a}^v$ for the residue $a+\mfm_v$ in $Kv$.

Assume that $\car(Kv)\neq 2$. A quadratic form $\vf$ over $K$ is called \emph{$v$-unimodular} if $\vf\simeq \la a_1,\dots,a_n\ra$ for $n=\dim(\vf)$ and certain $a_1,\dots,a_n\in\mg{\mc{O}}_v$.
In this case, the form $\la \ovl{a}^v_1,\dots,\ovl{a}_n^v\ra$ over $Kv$ is determined up to isometry by $\vf$, and it is called the \emph{residue form of $\vf$} (with respect to $v$) and denoted by $\ovl{\vf}^v$.

\begin{prop}\label{w-val-inequality}
We have $w(K)\geq w(Kv)$. 
 \end{prop}
\begin{proof}  
Consider an anisotropic, minimal weakly isotropic quadratic form $\phi$  over $Kv$. 
Set $n=\dim(\phi)$ and choose $a_1,\dots,a_n\in\mg{\mc{O}}_v$ such that 
$\phi\simeq \la \ovl{a}_1^{v},\dots,\ovl{a}_n^{v}\ra$.
Since $\phi$ is minimal weakly isotropic, there exist $t_1,\dots,t_n\in\mg{\mc{O}}_v$ which are sums of squares in $K$ and such that $\sum_{i=1}^n \ovl{\vphantom{t}a}^v_i \ovl{t}^v_i=0$,
whereby $\sum_{i=1}^n a_it_i\in\mfm_v$.
Let $a=\sum_{i=1}^nt_it_n^{-1}a_i$.
Then $-\ovl{a}^{v}=\ovl{a}_n^{v}$ in $Kv$.
By the choice of $a$, the quadratic form $\varphi=\la a_1,\dots,a_{n-1},-a\ra$ over $K$ is weakly isotropic.
Furthermore, $\vartheta$ is $v$-unimodular and $\phi\simeq \ovl{\vf}^v$.
We choose a weakly isotropic subform $\vartheta$ of $\vf$ with $\dim(\vartheta)\leq w(K)$.
Since $\ovl{\vf}^v=\phi$, which is anisotropic, it follows that $\vartheta$ is unimodular and $\ovl{\vartheta}^v$ is a weakly isotropic subform of $\phi$.
Since $\phi$ was minimal weakly isotropic, we obtain that $\ovl{\vartheta}^v\simeq \phi$ and hence $\dim(\phi)=\dim(\vartheta)\leq w(K)$.
This argument shows that $w(Kv)\leq w(K)$.
\end{proof}

 \begin{cor}\label{alg-rafufi-w-inequality}
  For every algebraic field extension $L/K$, we have 
   $$w(L)\leq w(K(X))\,.$$
\end{cor}
\begin{proof}
    It suffices to consider the case of a finite extension $L=K[X]/(p)$ where $p\in K[X]$ is a monic irreducible polynomial. 
    Then $L$ is the residue field of the $p$-adic valuation on $K(X)$.
    Hence $w(K(X))\geq w(L)$, by  \Cref{w-val-inequality}.
\end{proof}

The following is a variation of \cite[Prop.~5.5]{Bec06}.

\begin{prop}\label{hens-w-equal}
    Let $v$ be a henselian valuation on $K$.
    Then $w(K)=w(Kv)$.
\end{prop}
\begin{proof}
    Let $\vf$ be a weakly isotropic form over $K$. 
    It follows by Durfee's Theorem \cite[Theorem 3.11]{BDM25} that there exists $c\in\mg{K}$ and a $v$-unimodular quadratic form $\vartheta$ over $K$ such that $c\vartheta$ is a subform of $\vf$ and $\ovl{\vartheta}^v$ is weakly isotropic.
    Then $\ovl{\vartheta}^v$ contains a weakly isotropic subform of dimension at most $w(Kv)$. This subform is the residue form $\ovl{\psi}^v$ of a weakly isotropic form $\psi$ over $K$.
    Since $v$ is henselian, it follows again from Durfee's Theorem that $\psi$ is a subform of $\vartheta$.
    Then $c\psi$ is a weakly isotropic subform of $\vf$ of dimension at most $w(Kv)$. 
    This argument shows that $w(K)\leq w(Kv)$.
    The converse inequality is given by \Cref{w-val-inequality}.
\end{proof}

\section{Local-global principle for function fields in one variable}

Mehmeti's local-global principle from \cite{Meh19} for quadratic forms over function fields in one variable over a
complete rank-$1$ valued base field can be extended to quadratic spaces that are not necessarily finite-dimensional. 

Consider a valuation $v$ on $K$.
We denote by $\mathsf{Conv}(vK)$ the set of all proper convex subgroups of $vK$.
This set is totally ordered by inclusion.
We set $\rk(v)=|\mathsf{Conv}(vK)|\in\nat\cup\{\infty\}$ and call this the \emph{rank of $v$}.
We denote by $K^v$ the henselization
(see \cite[Sect.~5.2]{EP05}) and by $\hat{K}^v$ the completion with respect to $v$.
If $\rk(v)=1$, then $\hat{K}^v$ is henselian with respect to the natural extension of $v$, by \cite[Theorem 1.3.1]{EP05}, whereby $K^v\subseteq\hat{K}^v$. 
For a convex subgroup $\Delta$ of $vK$, 
composing $v$ with the quotient map $vK \rightarrow vK/\Delta$ defines a valuation
$v_1 : K\rightarrow (vK/\Delta)\cup\{\infty\}$ with $\mc{O}_v \subseteq\mc{O}_{v_1}$, also called a \emph{coarsening of $v$}.

We denote by $\mc{V}(K)$ the set of equivalence classes $[\nu]$ of valuations $\nu$ on $K$.
For a class $\gamma\in\mc{V}(K)$ of a valuation $\nu$, we set $\mc{O}_\gamma=\mc{O}_\nu$, $K\gamma=K\nu$ and we denote by $K^\gamma$ the henselization of $K$ with respect to $\nu$.
We endow $\mc{V}(K)$ with the constructible topology, defined as the coarsest topology such that the subsets $\{\gamma\in\mc{V}(K)\mid a\in\mc{O}_\gamma\}$ are open and closed for all $a\in\mg{K}$.
This turns $\mc{V}(K)$ into
a compact Hausdorff space; see \cite[Sect.~1]{HK94}.

A finitely generated field extension of transcendence degree $1$ is called a \emph{function field in one variable}.
Given a rank-$1$ valuation $v$ on $K$ and a function field in one variable $F/K$, 
we set $$\mc{V}(F/v)=\{\gamma\in\mc{V}(F)\mid \mc{O}_\gamma\cap K=\mc{O}_v\}\,.$$
Note that $\mc{V}(F/v)$ is closed in $\mc{V}(F)$, and hence it is a compact subspace.

The following is a variation of \cite[Cor.~7.2]{BDD25}.
Recall that we assume that 
$$\car(K)\neq 2\,.$$

\begin{thm}\label{Mehmeti-infinite}
    Let $v$ be a henselian valuation on $K$ with $\rk(v)=1$ and let $F/K$ be a function field in one variable. 
    Let $\vf$ be an anisotropic quadratic space over $F$ of dimension at least $3$. Then there exists a valuation $v'$ on $F$ extending $v$ such that $\vf_{F^{v'}}$ is anisotropic.
\end{thm}
\begin{proof}
    Assume first that $\dim(\vf)<\infty$.
    Hence $\vf$ is a quadratic form over $K$.
    By \cite[Theorem 4.4]{BDGMZ23}, there exists a valuation $\nu$ of rank $1$ on $F$ such that $\vf_{\hat{F}^\nu}$ is anisotropic and such that $\mc{O}_v\subseteq\mc{O}_\nu$. 
    Since $\rk(\nu)=1$, we have $F^\nu\subseteq\hat{F}^\nu$.
    Hence $\vf_{F^\nu}$ is anisotropic.
    If $\mc{O}_v=\mc{O}_\nu\cap K$, then we may rescale $\nu$ to achieve that $\nu|_K=v$, and then we take $v'=\nu$. 
    Assume now that 
    $\mc{O}_v\subsetneq \mc{O}_{\nu}\cap K$.
    Since $\rk(v)=1$, it follows that $K\subseteq\mc{O}_{\nu}$.
    Then the residue field $F\nu$ is a finite field extension of $K$. 
    Let $\nu^*$ denote the natural extension of $\nu$ to $F^\nu$. 
    Note that $(F^\nu)\nu^*=F\nu$.
    Since the valuation $v$ on $K$ is  henselian and $F\nu/K$ is a finite extension, $v$ extends uniquely to a valuation $v^\ast$ on $F\nu$, and this extension is again henselian with $\rk(v^\ast)=\rk(v)=1$.
     We compose $\nu^\ast$ with the valuation $v^\ast$ defined on $(F^\nu)\nu^\ast=F\nu$; see \cite[p.~45]{EP05}. This yields a rank-$2$ valuation $\mu$ on $F^\nu$ extending $v$ with 
     $\mc{O}_{\mu} = \{x \in \mc{O}_{\nu^\ast} \mid \bar{x}^{\nu^\ast} \in \mc{O}_{v^\ast}\}$.
     By \cite[Cor.~4.1.4]{EP05}, since $\nu^\ast$ and $v^\ast$ are henselian, so is $\mu$.
     Therefore $F^\nu$ is equal to the henselization $F^{v'}$ of the restriction $v'=\mu|_F$.
     Hence
     $\vf_{F^{v'}}$ is anisotropic.

This proves the statement in the case where $\dim(\vf)<\infty$.
Consider now the general case. Write $\vf=(V,q)$ and set $\mc{V}=\mc{V}(F/v)$.
We denote by $\mc{U}$ the set of finite-dimensional $K$-subspaces of $V$.
With $U\in \mc{U}$, we associate the quadratic form $\vf_{U}=(U,q|_{U})$ over $F$ and the set $\mc{V}_U=\{\gamma\in \mc{V
} \mid (\vf_U)_{F^{\gamma}}   \mbox{ isotropic}\}$. 
For $U,U'\in\mc{U}$ with $U\subseteq U'$, we have that $\vf_U$ is a subform of $\vf_{U'}$, whereby $\mc{V}_{U}\subseteq \mc{V}_{U'}$. 
For $U\in\mc{U}$, we have by \cite[Lemma 8.1]{BDD25} that $\mc{V}_U$ is open in $\mc{V}$. 

Assume now that $\vf_{F^{\gamma}}$ is isotropic for every $\gamma\in \mc{V}(F/v)$. 
Then for every $\gamma\in \mc{V}$ there exists $U\in \mc{U}$ such that $(\vf_U)_{F^\gamma}$ is isotropic. 
Hence $\mc{V}=\bigcup_{U\in \mc{U}}\mc{V}_U$. 
Since $\mc{V}$ is compact, there exist $n\in\nat$ and $U_1, \dots, U_n\in \mc{U}$ such that $\mc{V}=\bigcup_{i=1}^n\mc{V}_{U_i}$. 
For $U=U_1+\dots+U_n$, we obtain that $U\in \mc{U}$ and $\mc{V}=\mc{V}_U$.
Hence $(\vf_U)_{F^\gamma}$ is isotropic for every $\gamma\in\mc{V}$.
Since $\dim(\vf_U)\geq 3$, the first part of the proof yields that $\vf_U$ is isotropic.
Consequently, $\vf$ is isotropic.
\end{proof}

\begin{cor}

    Let $v$ be a henselian valuation on $K$ with $\rk(v)=1$.
Let $F/K$ be a function field in one variable.
Let $\vf$ be a quadratic space of dimension at least $3$ over $F$. Then $\vf$ is isotropic if and only if $\vf_{\hat{F}^\nu}$ is isotropic for every rank-$1$ valuation $\nu$ on $K$ with $\mc{O}_v\subseteq\mc{O}_\nu$.

\end{cor}
\begin{proof}
   
If $\vf$ is isotropic, then $\vf_{\hat{F}^\nu}$ is isotropic for every valuation $\nu$ on $F$.
Assume now that $\vf$ is anisotropic over $F$.
By \Cref{Mehmeti-infinite}, there exists $\gamma=[\eta]\in \mc{V}(F/v)$ such that $\vf_{F^{\gamma}}$ is anisotropic.

Then $1\leq \rk(\eta)\leq 2<\infty$, whereby $\eta F$ has a  maximal proper convex subgroup $\Delta$. 
Let $\nu$ denote the rank-$1$ valuation on $F$ obtained as the coarsening of $\eta$ with value group $\eta F/\Delta$. 
Then $\rk(\nu)=1$ and $\mc{O}_v\subseteq\mc{O}_\eta\subseteq \mc{O}_\nu$.
Let $\eta^\ast$ denote the natural extension of $\eta$ to $F^\eta$.
Note that $\eta^\ast$ is henselian and its value group is $\eta F$.
Its rank-$1$ coarsening $\nu^\ast$ with value group $\eta F/\Delta$ extends $\nu$
 and, by \cite[Cor.~4.1.4]{EP05}, it is henselian on $F^\eta$.
Therefore $F^\nu\subseteq F^\eta$, and we conclude that $\vf_{F^\nu}$ is anisotropic.
Since $F^\nu$ is dense in $\hat{F}^\nu$, it  follows by \cite[Lemma 3.16]{Meh19} that $\vf_{\hat{F}^\nu}$ is anisotropic.
\end{proof}

\section{Weakly isotropic forms over function fields in one variable}

We are ready to study the $w$-invariant for function fields in one variable over a henselian valued field. 

\begin{prop}\label{LG-strongly-balanced-Pfister}
    Let $v$ be a henselian valuation on $K$ with $\rk(v)=1$.
    Let $F/K$ be a function field in one variable and $\pi$ a Pfister form over $F$.
    Then $\pi$ is strongly balanced if and only if $\pi_{F^\gamma}$ is strongly balanced for every $\gamma\in\mc{V}(F/v)$.
\end{prop}

\begin{proof}
For $n\in\nat$ we set $\vf_n=n\times\la 1\ra\perp \pi$.
Let $\vf$ be the quadratic space over $K$ given as the direct limit of $(\vf_n)_{n\in\nat}$.
For any field extension $F'/F$,
we have by \Cref{Pfister-balanced} that $\pi_{F'}$ is strongly balanced if and only if $(\vf_n)_{F'}$ is isotropic for some $n\in\nat$, which is if and only if $\vf_{F'}$ is isotropic.
Now the conclusion follows by \Cref{Mehmeti-infinite}.
\end{proof}

\begin{cor}\label{LG-w2}

    Let $v$ be a henselian valuation on $K$ with $\rk(v)=1$ and let $F/K$ be a function field in one variable.
    Assume that $w(F^\gamma)\leq 2$ for every $\gamma\in\mc{V}(F/v)$.
    Then $w(F)\leq 2$.
\end{cor}
\begin{proof}
    Consider a torsion $2$-fold Pfister form $\pi$ over $F$.
    It follows by \Cref{Arason-Pfister} from the hypothesis that $\pi_{F^\gamma}$ is strongly balanced for every $\gamma\in\mc{V}(F/v)$.
    Hence, according to \Cref{LG-strongly-balanced-Pfister}, $\pi$ is strongly balanced over $F$.
    In view of \Cref{Arason-Pfister}, this argument shows that $w(F)\leq 2$.
\end{proof}

We set 
$$\wt{w}(K)=\sup \{w(F)\mid F/K \mbox{ is a function field in one variable}\}\in \nat \cup \{\infty\}\,.$$

\begin{prop}\label{tilde}
   Let $v$ be a valuation on $K$ with $\car(Kv)\neq 2$. Then 
   $$\wt{w}(K)\geq\wt{w}(Kv)\,.$$
\end{prop}
\begin{proof}
  Consider a function field in one variable $F_0/Kv$. By  \cite[Lemma 5.6]{BDGMZ23}, there exists a function field in one variable $F/K$ and an extension $v$ to a valuation $v'$ on $F$ such that $Fv'$ is $Kv$-isomorphic to $F_0$.
  By \Cref{w-val-inequality}, we obtain that $w(F)\geq w(Fv')$. 
  Hence $w(F_0)=w(Fv')\leq w(F)\leq \wt{w}(K)$. Having this for all function fields in one variable $F_0/Kv$, we conclude that $\wt{w}(Kv)\leq \wt{w}(K)$.  
\end{proof}

\begin{thm}\label{value}
    Let $v$ be a henselian valuation on $K$.
    Assume that  $\car(Kv)\neq 2$ and $\wt{w}(Kv)\leq 2$.
    Then $$\wt{w}(K)=\wt{w}(Kv)\,.$$
    
\end{thm}
\begin{proof}
Suppose first that $\wt{w}(Kv)=1$. Then $Kv$ is nonreal.
Since $v$ is henselian, it follows that $K$ is nonreal.
Hence every field extension $K'/K$ is nonreal, whereby $w(K')=1$.
In particular $\wt{w}(K)=1=\wt{w}(Kv)$ in this case.

We may now assume that $\wt{w}(Kv)=2$.
Then $w(E)=2$ for some function field in one variable $E/Kv$, and in particular $E$ is real, whereby $Kv$ is real.
It follows that $K$ is real.
Then $K(X)$ is real, so we obtain that $\wt{w}(K)\geq w(K(X))\geq 2$.
It therefore remains to show that $\wt{w}(K)\leq 2$.

We first consider the case where $\rk(v)=1$.
Consider a function field in one variable $F/K$.
Recall that, f or $\gamma=[\nu]\in\mc{V}(F/v)$, we set $F\gamma=F\nu$ and denote the henselisation
$F^\nu$ by $F^\gamma$.

Consider an arbitrary $\gamma\in\mc{V}(F/v)$. We claim that $w(F^\gamma)\leq 2$.
Let $\nu$ be a valuation on $F$ with $\gamma=[\nu]$.
Assume first that $\nu|_K$ is trivial. Then $F\gamma$ is a finite extension of $K$.
As $v$ is henselian, it extends uniquely to a valuation $v'$ on $F\gamma$.
Then $v'$ is henselian, so $w(F\gamma)=w((F\gamma)v')$, by \Cref{hens-w-equal}.
Since $(F\gamma)v'$ is a finite field extension of $Kv$, we have $w((F\gamma)v')\leq w(Kv(X))$, by \Cref{alg-rafufi-w-inequality}.
Since $Kv(X)/Kv$ is a function field in one variable,
the hypothesis implies that $w(Kv(X))\leq 2$.
Therefore  by \Cref{alg-rafufi-w-inequality}, we have $w(F\gamma)=w((F\gamma)v')\leq w(Kv(X))\leq 2$.
Assume now that $\nu|_K$ is nontrivial. Then $\nu|_K$ is equivalent to $v$.
This implies that $F\gamma/Kv$ is either a function field in one variable or an algebraic extension. 
In view of \Cref{alg-rafufi-w-inequality}, we obtain in either case from the hypothesis on $Kv$ that $w(F\gamma)\leq 2$.
By \Cref{hens-w-equal}, we thus have $w(F^\gamma)=w(F\gamma)\leq 2$.

We have thus shown that $w(F^\gamma)\leq 2$
for all $\gamma\in\mc{V}(F/v)$. Since $\rk(v)=1$, we conclude by \Cref{LG-w2} that $w(F)\leq 2$.
This completes the proof that $\wt{w}(K)\leq 2$ in the case where $\rk(v)=1$.

We will now extend this 
to the situation where $n=\rk(v)<\infty$.
We will prove by induction on $n$ that $\wt{w}(K)\leq 2$.
If $n=0$, then $Kv=K$, and hence the statement is trivial.
Assume now that $1\leq n<\infty$. 
We will view $v$ as a composition of two valuations of rank $1$ and rank $n-1$, respectively.
Since the value group $vK$ is nontrivial and of finite rank, it has a unique maximal proper convex subgroup $\Delta$. 
Let $v_1$ denote the valuation on
$K$ which is the coarsening of $v$ with value group $vK/\Delta$.
Then $v_1$ is a rank-$1$ valuation on $K$.
Furthermore $v$ induces a valuation $\eta$ on the residue field $K{v_1}$ with value
group $\eta (K{v_1}) = \Delta$, and with residue field $(K{v_1}) \eta = Kv$. In particular, $\rk(\eta)=n-1$.
By \cite[Cor.~4.1.4]{EP05} the valuations $v_1$ on $K$ and $\eta$ on $K{v_1}$
are both henselian. 
Since $(Kv_1)\eta=Kv$, $\wt{w}(Kv)\leq 2$ and $\rk(\eta)=n-1$, the induction hypothesis yields that $\wt{w}(Kv_1)\leq 2$.
Since $\rk(v_1)=1$, the  case which is already established then yields that 
$\wt{w}(K)\leq 2$.

We will now prove that $\wt{w}(K)\leq 2$ without assuming that $\rk(v)<\infty$.
Consider a function field in one variable $F/K$.
We need to show that $w(F)\leq 2$.
To this aim, by \Cref{Arason-Pfister} it suffices to show that every torsion $2$-fold Pfister form over $F$ is strongly balanced.
Consider a torsion $2$-fold Pfister form $\pi$ over $F$.
We choose a finite subset $S$ of $F$ such that~$\pi$ is defined as a $2$-fold Pfister form over $\qq(S)$ and is a torsion form over $\qq(S)$, and further such that
    $F=K(S)$ and $\qq(S)/\qq(S\cap K)$ is of transcendence degree $1$.
    We set $K_0=\qq(S\cap K)$. 
    Since the absolute transcendence degree of $K_0$ is bounded by $|S|$, it follows by \cite[Theorem~3.4.3]{EP05} that $\rk(v|_{K_0})\leq |S|+1<\infty$.
    By \cite[Lemma 2.3]{BDM25}, there exists a subfield $K_1$ of $K$ containing $K_0$ such that $vK_1/vK_0$ is torsion, $v|_{K_1}$ is henselian, and $Kv/K_1v$ is algebraic and purely inseparable.
    Since $\car(Kv)=0$, the last simply means that $Kv=K_1v$.
    Since $vK_1/vK_0$ is torsion, we have bijection $\mathsf{Conv}(vK_1)\to\mathsf{Conv}(vK_0),\Delta\mapsto \Delta\cap vK_0$, whereby $\rk(v|_{K_1})=\rk(v|_{K_0})<\infty$.
    Using the finite-rank case which we already settled, we obtain that $\wt{w}(K)=\wt{w}(K_1v)=\wt{w}(Kv)=2$.
    Let $F_1=K_1(S)$.
    Then $F_1/K_1$ is a function field in one variable and $\pi$ is is defined over $F_1$ and it is a torsion form over $F_1$.
    Hence $w(F_1)\leq \wt{w}(K_1)=2$ and we conclude that $\pi$ contains a binary torsion form over $F_1$, and hence also over $F$.
    Hence, in view of \Cref{Pfister-balanced}, $\pi$ is strongly balanced.
    This finishes the proof.
\end{proof}

\begin{qu}
Does $\wt{w}(K)=\wt{w}(Kv)$ hold in general if $v$ is a henselian valuation $v$ on $K$ with $\car(Kv)\neq 2$?
\end{qu}

We turn to the case where $K$ is hereditarily pythagorean.
By a result of E. Becker \cite[Chap.~III, Theorem 4]{Bec78}, this is equivalent to having that $K$ is real and $p(K(X))=2$.
In particular, this implies that $w(K(X))= 2$, by \Cref{wp}.

\begin{thm}\label{main}
    If $K$ is hereditarily pythagorean, then $\wt{w}(K)=2$.
    
\end{thm}
\begin{proof}
By definition, $K$ is real, so we have $\wt{w}(K)\geq w(K(X))\geq w(K)\geq 2$, and it remains to show that $\wt{w}(K)\leq 2$.

Since $K$ is hereditarily pythagorean, by Br\"ocker's Theorem \cite[Prop.~3.5]{Bro76}, there exist a henselian valuation $v$ such that the residue field $Kv$ is hereditarily pythagorean and $|\scg{(Kv)}|\leq 4$. 
  For a function field in one variable $E/Kv$, it follows by \cite[Theorem 6.6]{BDGMZ23} that $\I^3_\tors E=0$, and by \Cref{In-tors-w}, this implies that $w(E)\leq 2$. 
  This argument shows that $\wt{w}(Kv)\leq 2$.
  We conclude by \Cref{value} that $\wt{w}(K)\leq 2$. 
\end{proof}

By \cite[Theorem 1.6]{BDGMZ23}, if $K$ is hereditarily pythagorean, then $p(F)\leq 5$ for every function field in one variable $F/K$, and examples where $p(F)=3$ are known.
Hence \Cref{main} cannot be obtained via \Cref{wp}.

\begin{qu}
Let $K$ be hereditarily pythagorean and $F/K$ a function field in one variable.
Can \Cref{AP:B} be applied to $F$ with $r=2$? 
Equivalently, for arbitrary $t_1,t_2\in\mg{(\mathsf{\Sigma}F^2)}$, does there exist $t\in\mg{(\mathsf{\Sigma}F^2)}$ with $t_1,t_2\in\D\la 1,t\ra$?
\end{qu}  

\subsection*{Acknowledgments}
This work was supported by the \emph{Bijzonder Onderzoeksfonds, University of Antwerp} (project \emph{BOF Opvang MSCA IF}, ID 51418). We further would like to express our gratitude to Nicolas Daans for helping us with making the proof of \Cref{Mehmeti-infinite} correct.

\section*{Declarations}

\subsection*{Data availability}
There is no further associated data.

\subsection*{Conflict of interest} The authors declare that there is no conflict of interest.

\bibliographystyle{plain}

\end{document}